\documentclass{amsart}
\usepackage{amsmath, amssymb,amsthm}
\usepackage{graphicx}
\usepackage{enumitem}
\usepackage{color}

\newtheorem{theorem}{Theorem}[section]
\newtheorem{corollary}[theorem]{Corollary}
\newtheorem{lemma}[theorem]{Lemma}
\newtheorem{proposition}[theorem]{Proposition}
\newtheorem{definition}[theorem]{Definition}

\theoremstyle{remark}
\newtheorem*{remark}{Remark}
\newtheorem*{example}{Example}

\begin{document}

%\title[The volume of intersections of two horospheres]{{\red{The volume of intersections of two horospheres in an asymptotically harmonic manifold}}}%on asymptotic harmonic manifolds}% in terms of Busemann functions}
\title[Two theorems on the intersections of horospheres]{Two theorems on the intersections of horospheres in asymptotically harmonic spaces}
\author{Sinhwi Kim}
\address{Department of Mathematics, Sungkyunkwan University \\
Suwon, 16419, Korea}
\email{kimsinhwi@skku.edu, parkj@skku.edu}
%\email{parkj@skku.edu}

\author{JeongHyeong Park}
%\address{Department of Mathematics, Sungkyunkwan University, \\
%Suwon, 16419, Korea}
%\email{parkj@skku.edu}

%\subjclass[2010]{53C30}
%\keywords{asymptotic harmonic, horosphere, Busemann function, visibility}


\maketitle
%{\red Insert subjclass[2010]{53C30}\\
%keywords{asymptotic harmonic, horosphere, Busemann function, visibility ?}}

\begin{abstract}
We use Busemann functions to construct volume preserving mappings in an asymptotically harmonic manifold.
If the asymptotically harmonic manifold satisfies the visibility condition, we construct mappings which preserve distances in some directions.
We also prove that some integrals on the intersection of horospheres are independent of the differences between the values of the corresponding Busemann functions and we establish an upper bound of the volume of the intersection of two horospheres which is independent
of the difference between values of corresponding Busemann functions.
\end{abstract}

\section{Introduction}

%History and background

% asymptotically harmonicity and results

Let $(M,g)$ be a Hadamard manifold, that is, simply connected, complete Riemannian manifold of non-positive sectional curvatures. Let $d(p,q)$ be the distance between $p,q\in M$. For each unit tangent vector $v$ to $M$,
the Busemann function $b_v: M\rightarrow\mathbb{R}$ on $(M,g)$ is defined by
\begin{equation*}
b_v(x)=\lim_{t\rightarrow\infty}(d(x,\gamma_v(t))-t),
\end{equation*}
where $\gamma_v: [0,\infty)\rightarrow M$ is a geodesic ray such that $\gamma_v'(0)=v$.
%On Hadamard manifold the distance function is convex.
Busemann functions on a Hadamard manifold are convex \cite{Esc} and $C^2$
\cite{HI}.
The level hypersurfaces of Busemann functions are called horospheres.
A Riemannian manifold is called harmonic, if, about any point, the geodesic spheres
of sufficiently small radii are of constant mean curvature.
% Let (M,g) denote a complete Riemannian manifold.
% (i) M is called a harmonic space, if about any point the geodesic spheres
% of sufficiently small radii are of constant mean curvature.872 J. HEBER GAFA
% (ii) M is called asymptotically harmonic, if M has no conjugate points
% and the mean curvature of its horospheres is constant.
A Hadamard manifold $(M,g)$ is called asymptotically harmonic if the mean curvatures of its horospheres are constant \cite{Heb}. Asymptotically harmonic manifolds belong to a special class of Riemannian manifolds.
All harmonic manifolds are asymptotically harmonic, and all known asymptotically harmonic manifolds are harmonic and homogeneous \cite{Heb}.
It is also remarkable that the following question is remained open: Is a harmonic manifold homogeneous?

To consider the homogeneousness of a harmonic manifold and, more generally, of an asymptotically harmonic manifold,
we construct a volume preserving mapping which maps one point to another point.
For this construction, we use the mean curvature of horospheres and the variation of the volume density for the flow in the orthogonal direction to horospheres.

% asymptotic geodesics and Busemann functions
%Let $d(p,q)$ be the distance between $p,q\in M$, and

%We denote the distance between them in a Riemannian manifold $(M,g)$ by $d(p,q)$,
 %for all points $p,q\in M$.

% \begin{definition}\cite[268p]{BrH}
% %Let $(M,g)$ be a complete, simply connected manifold without conjugate points.
% Let $(M,g)$ be a Hadamard manifold.
% For each unit tangent vector $v$ to $M$, define a Busemann function $b_v: M\rightarrow\mathbb{R}$ defined by
% \begin{equation*}
% b_v(x)=\lim_{t\rightarrow\infty}(d(x,\gamma_v(t))-t),
% \end{equation*}
% where $\gamma_v: [0,\infty)\rightarrow M$ is a geodesic ray such that $\gamma_v'(0)=v$. The level sets of Busemann functions are called horospheres.
% \end{definition}

% \begin{definition}\cite{Heb} % heb?
% %Let $(M,g)$ be a complete, simply connected manifold without conjugate points.
% Let $(M,g)$ be a Hadamard manifold.
% $(M,g)$ is called asymptotically harmonic if the mean curvature of horospheres is constant.
% \end{definition}

% Description of this articlez result

Let $C_c(M)$ be the set of continuous functions with compact support and $T^1M$ be a unit tangent bundle.
% For asymptotically harmonic manifolds, we {{consider}} some ``volume preserving" mapping from the Busemann %functions.
We obtain the following theorem.
\begin{theorem}\label{thm_volpre}
Let $(M,g)$ be an asymptotically harmonic manifold. Then, for all points $p\neq q\in M$, there exists a diffeomorphism $F:M\rightarrow M$ such that $F(p)=q$ and
\begin{equation*}
\int_M f(x)d\mu(x) = \int_M f(F(x))d\mu(x),
\end{equation*}
for all $f\in C_c(M)$, where $d\mu$ is the Riemannian measure on $M$.
\end{theorem}

%...
Szab\'o \cite{Sza} proved that a Riemannian manifold is a harmonic manifold if and only if
the volume of the intersection of geodesic spheres  depends only on the radii and the distance between the centers of the geodesic spheres.
Csik\'os and M. Horv\'ath \cite{CsH, CsH2} proved that the intersections can be restricted to the cases with same radii, and they
also proved that the volume of a tubular neighborhood about a geodesic depends only
on the length of the geodesic and the radius if and only if the Riemannian manifold is harmonic.
In addition, some relations between integrals and measures on a harmonic manifold and its ideal boundary were found by Itoh and Satoh \cite{ItS} and by Rouvi\'ere \cite{Rou}. {{Knieper and Peyerimhoff \cite{KnP} also considered the integrals and measures on harmonic manifolds to find a solution of the Dirichlet problem at infinity,}} and Biswas, Knieper, and Peyerimhoff \cite{BKP} proved that there exists a Fourier transform between harmonic manifolds and its ideal boundary.

% def of visibility
 A Hadamard manifold $(M, g)$ is called a visibility manifold, or satisfies the visibility condition,
 if, for any two different points $v_1$, $v_2$ at infinity, there is a
geodesic $\gamma : \mathbb{R} \rightarrow M$ with $\gamma(\infty) = v_1$, $\gamma(-\infty) = v_2$
(see section 2). % \cite{BGS}).
All intersections of two horospheres in a Hadamard manifold are bounded if and only if the Hadamard manifold is a visibility manifold (see Definition \ref{def_visi} and Lemma \ref{lem_visibility}).
We note that the volume of an intersection of two horospheres in a visibility manifold is finite.  % also ..
% Also, a Riemannian manifold is called harmonic if the mean curvatures of geodesic spheres depends only on the radius. It is equivalent to that the volume of the intersection of geodesic spheres with same radii depends only on the radius and the distance between centers of the geodesic spheres \cite{CsH}.
In particular, for every harmonic manifold which satisfies the visibility condition, the volume of the intersection $b_{v_1}^{-1}(c_1)\cap b_{v_2}^{-1}(c_2)$ of two horospheres is independent of $c_1-c_2$.
%Similarly, for asymptotically harmonic manifolds, by using a mapping between the intersections for two Busemann functions, we have the following theorem and corollary:
For asymptotically harmonic manifolds which satisfy the visibility condition (i.e., asymptotically harmonic, visibility manifolds),
in this article, we prove that some integrals on the intersection of two horospheres
are independent of the difference between values of corresponding Busemann functions.
We also obtain an upper bound of the volume of the intersection of two horospheres.
Throughout this paper, we assume that the dimension of the manifold is $n\geq 2$.
Our main theorem is the following:

\begin{theorem}\label{thm_volz_bound}
Let $(M,g)$ be an asymptotically harmonic, visibility manifold.
 Let $p\in M$, $v_1\neq v_2\in T_p^1M$, and $c\in\mathbb{R}$.
%Suppose that $(M,g)$ is a visibility manifold.
Then there exists a constant $C>0$ such that
the $(n-2)$-dimensional volume of the intersection $b_{v_1}^{-1}(c-t)\cap b_{v_2}^{-1}(c+t)$ is less than $C$, for all $t\in\mathbb{R}$.
\end{theorem}

\section{Preliminary}

The integration of a function on a connected Riemannian manifold $(M,g)$ can be computed in terms of the integrals on hypersurfaces:

\begin{proposition}\cite{Cha}\label{prop_cha}
Let $(M,g)$ be a connected Riemannian manifold.
Let $\varphi:M\rightarrow\mathbb{R}$ be a $C^1$ function such that $\nabla\varphi$ is non-vanishing on $M$, and let $S_t$ be the hypersuface defined by $S_t=\{x\in M:\varphi(x)=t\}$, for all $t\in\mathbb{R}$. Then, for all $f\in C_c(M)$ and $t\in\mathbb{R}$,
\begin{equation*}
\int_M f(x)d\mu(x)=\int_{\mathbb{R}}\int_{S_t}\frac{f(x)}{\|\nabla\varphi(x)\|}d\mu_t(x)dt,
\end{equation*}
where $d\mu$ is the Riemannian measure on $M$ and $d\mu_t$ is the induced Riemannian measure on $S_t$.
\end{proposition}

%=================================================

We introduce an infinitesimal ``volume preserving" mapping on $(M,g)$ in \cite{HaM}.
%\begin{definition}\cite{HaM}
%Let $d\mu$ be a volume form on a Riemannian manifold $(M,g)$ and $X$ a vector field on $M$.
We say that a vector field $X$ on $M$ is volume preserving if
%\begin{equation*}
$L_X(d\mu)=0$,
%\end{equation*}
where $L_X$ is the Lie derivative with respect to $X$.
If $\phi_t$ is the flow generated by $X$, then we call $\phi_t$ volume preserving if $X$ is volume preserving.
In that case, we have
%\begin{equation*}
$(\phi_t)^*(d\mu)=d\mu.$
%\end{equation*}
%\end{definition}
This flows preserve integrals on a Riemannian manifold $(M,g)$ since it preserves the Riemannian measure if it is a diffeomorphism.

\begin{proposition}\label{prop_volprez}
Let $\phi_t$ be a flow on a Riemannian manifold $(M,g)$ and suppose that $\phi_t$ is a diffeomorphism on $M$ for all $t\in\mathbb{R}$. Then {{$\phi_t$}} is volume preserving if and only if
\begin{equation}\label{eq_inte}
\int_M f(x) d\mu(x) = \int_M f(\phi_t(x)) d\mu(x),
\end{equation}
for all $f\in C_c(M)$, where $d\mu$ is the Riemannian measure on $M$.
\end{proposition}

\begin{proof}
If $\phi_t$ is volume preserving, then, clearly, equation \eqref{eq_inte} holds.
Now, suppose that equation \eqref{eq_inte} holds for all $f\in C_c(M)$. Then
\begin{equation}\label{eq_volFtprop}
\int_M f(\phi_t(x))d\mu(x)=\int_M f(x)d\mu(x)=\int_M f(\phi_t(x))(\phi_t)^*(d\mu)(x)
\end{equation}
for all $f\in C_c(M)$. We denote $(\phi_t)^*(d\mu)=f_0 \: d\mu$ for some function $f_0$. Then, by equation \eqref{eq_volFtprop}, $f_0=1$ and $(\phi_t)^*(d\mu)=d\mu$.
\end{proof}

%=======================================

Now, we define asymptotic geodesic rays:

\begin{definition}\cite{BrH}
%Let $(M,g)$ be a complete, simply connected manifold without conjugate points,
Let $(M,g)$ be a Hadamard manifold, and
$\gamma_1,\gamma_2:[0,\infty)\rightarrow M$ geodesic rays. $\gamma_1,\gamma_2$ are said to be asymptotic if there exists a constant $C>0$ such that
\begin{equation*}
d(\gamma_1(t),\gamma_2(t))\leq C,
\end{equation*}
for all $t\geq 0$.
\end{definition}

This gives
an equivalence relation on geodesic rays: two geodesic rays are equivalent if and only if they are asymptotic.
The set $\partial_{\infty} M$ of points at infinity is the set of equivalence classes of this relation. %ideal boundary.
%and called an ideal boundary of (M, g).
%The classes consisting of asymptotic geodesic rays construct the points at infinity.
This is also called the ideal boundary of $(M,g)$.
% We note that $\theta\in\partial_{\infty}M$, and $x\in M$.
% Then there exists a unique geodesic ray $\gamma\in\theta$ from $x$ \cite[261p]{BrH}.
% \begin{definition}\cite[260p]{BrH}  % \cite[22p]{BGS},
% Let $(M,g)$ be a Hadamard manifold.
% Consider the equivalence relation on geodesic rays: two geodesic rays are equivalent if and only if they are asymptotic.
% The set $\partial_{\infty} M$ of points at infinity is the set of equivalence classes of this relation. %ideal boundary.
% \end{definition}
Consequently, there exists a set of distance functions from each point and their limits.
%\begin{definition}\cite[267p]{BrH}
Let $C(M)$ be the space of continuous functions on $M$ equipped with the topology of uniform convergence on bounded subsets.
Let $C_*(M)$ denote the quotient of $C(M)$ by the 1-dimensional subspace of constant functions.
%We denote the image in $C_*(M)$ of $f$ by $\bar{f}$, for all $f\in C(M)$.
%\end{definition}
% \begin{definition}\cite[267p]{BrH}
% Let $(M,g)$ be a Riemannian manifold, $C(M)$ the space of continuous functions on $M$ equipped with the topology of uniform convergence on bounded subsets.
% Let $C_*(M)$ denote the quotient of $C(M)$ by the 1-dimensional subspace of constant functions. We denote the image in $C_*(M)$ of $f$ by $\bar{f}$, for all $f\in C(M)$.
% \end{definition}
For all $v\in T^1M$ and $t\in\mathbb{R}$,
\begin{equation*}
\nabla b_v(\gamma(t))=-\gamma'(t),
\end{equation*}
where $\gamma$ is a geodesic ray in $(M,g)$ asymptotic to $\gamma_v$. In particular, the image of Busemann functions in $C_*(M)$ can be associated to geodesic rays. The points at infinity also corresponds to the images of Busemann functions in $C_*(M)$.

\begin{proposition}\cite{BrH}
Let $(M,g)$ be a Hadamard manifold.
Then the Busemann functions associated to asymptotic geodesic rays in $M$ are equal up to addition of a constant.
\end{proposition}

\begin{proposition}\cite{BrH}\label{prop_unique} % Hadamard => complete CAT(0)   ... conversely?
Let $(M,g)$ be a Hadamard manifold, $\theta\in\partial_{\infty}M$, and $x\in M$.
Then there exists a unique geodesic ray $\gamma\in\theta$ from $x$.
\end{proposition}

By the propositions, points at infinity bijectively correspond to images of Busemann functions in $C_*(M)$.
%For every $\theta\in\partial_{\infty}M$, write a Busemann function associated to $\theta$ by $b_{\theta}$.

For two elements of $\partial_{\infty}M$, there exists a geodesic from one direction to another direction. If such geodesic exists for all pairs of distinct elements of $\partial_{\infty}M$, then we call such a Hadamard manifold a visibility manifold:

\begin{definition}\cite{BGS}\label{def_visi}
A Hadamard manifold $(M,g)$ is called a visibility manifold if, for all $p\in M$ and $v_1\neq v_2\in T_p^1M$,
there exists a geodesic ray $\gamma$ such that $\gamma$ is a asymptotic to $\gamma_{v_1}$, and the geodesic ray
$t\mapsto \gamma(-t)$ is asymptotic to $\gamma_{v_2}$.
\end{definition}
Such geodesic ray $\gamma$ is said to be bi-asymptotic to $v_1,v_2$. There exist several equivalent conditions for the visibility, one of which is as follows.
\begin{lemma}\cite{BGS}\label{lem_visibility}
Let $(M,g)$ be a Hadamard manifold. Then the following statements are equivalent:
\begin{enumerate}[label=(\roman*)]
  \item $(M,g)$ is a visibility manifold.
  \item $b_{v_1}^{-1}((-\infty, c_1))\cap b_{v_2}^{-1}((-\infty,c_2))$ is bounded for all $p\in M$, $v_1\neq v_2\in T_p^1M$, and $c_1,c_2\in\mathbb{R}$.
\end{enumerate}
\end{lemma}
For example, if a Hadamard manifold $(M,g)$ satisfies the curvature condition $K\leq -a^2<0$, for some $a\in\mathbb{R}$, then $(M,g)$ is a visibility manifold \cite{BGS}.
%========
In a Hadamard manifold, bi-asymptotic geodesics are normal geodesics of some intersection of horospheres of the form $b_v^{-1}(0)\cap b_{-v}^{-1}(0)$:
\begin{proposition}\cite{Esc}\label{prop_esc_from_ref}
Let $(M,g)$ be a Hadamard manifold. Then, for every $v\in T^1M$,
$b_v^{-1}(0)\cap b_{-v}^{-1}(0)$ is connected,
\begin{equation*}
\nabla b_v(x)+\nabla b_{-v}(x)=0,
\end{equation*}
for all $x\in b_v^{-1}(0)\cap b_{-v}^{-1}(0)$, and the geodesics which is asymptotic to $\gamma_v$ and intersects $b_v^{-1}(0)\cap b_{-v}^{-1}(0)$ orthogonally at a point are bi-asymptotic to $v,-v$.
\end{proposition}
Consequently, for two distinct bi-asymptotic geodesics, there exists a 2-dimensional flat, totally geodesic embedding containing them:
\begin{theorem}\cite{Esc}\label{thm_esc_from_ref}
Let $(M,g)$ be a Hadamard manifold. Then, for all $v\in T^1M$,
the Busemann function $b_v$ is convex, and the set $b_v^{-1}(0)\cap b_{-v}^{-1}(0)$ is convex.
If two geodesics $\gamma_1,\gamma_2$ are bi-asymptotic, there exists $a>0$ and a totally geodesic, isometric embedding
$F:[0,a]\times\mathbb{R}\rightarrow M$ such that $\gamma_1=\left.F\right|_{\{0\}\times\mathbb{R}}$ and $\gamma_2=\left.F\right|_{\{a\}\times\mathbb{R}}$.
\end{theorem}
Let $p\in M$, $v_1\neq v_2\in T_p^1M$, and
\begin{equation*}
D=\{x\in M: \nabla b_{v_1}(x)+\nabla b_{v_2}(x)=0\}.
\end{equation*}
Note that $D$ is closed and $D$ is nonempty if $(M,g)$ is a visibility manifold.
For every point $x\in D$,
by Proposition \ref{prop_unique}, there exists a unique geodesic ray from $x$ which is asymptotic to $v_1$, so it is bi-asymptotic to $v_1,v_2$.
In particular, by Theorem \ref{thm_esc_from_ref}, $D$ is connected.
Suppose that $D$ is nonempty, so $D$ contains a bi-asymptotic geodesic to $v_1,v_2$.
Let $c_0$ be the constant value of $b_{v_1}+b_{v_2}$ on $D$. If $c_1,c_2\in\mathbb{R}$ and $c_1+c_2=c_0$, then
there exists a point $x$ on each bi-asymptotic geodesic such that $b_{v_1}(x)=c_1$ and $b_{v_2}(x)=c_2$, and
\begin{equation}\label{eq_intersection_in_D}
b_{v_1}^{-1}(c_1)\cap b_{v_2}^{-1}(c_2)=b_v^{-1}(0)\cap b_{-v}^{-1}(0)\subseteq D,
\end{equation}
where $v=\nabla b_{v_1}(x)$.

\section{Proof of Theorem \ref{thm_volpre}}

%\begin{proof}[]
Let $p,q\in M$.
Since $(M,g)$ is a Hadamard manifold, there exists a unit-speed geodesic $\gamma$ such that $\gamma(0)=p$, $\gamma(t_0)=q$ for some $t_0>0$. Set $v=\gamma'(0)$.
Let $S_t=b_v^{-1}(t)$ be the level set of $b_v$ for all $t\in\mathbb{R}$.
Consider a diffeomorphism $\phi:\mathbb{R}\times S_0\rightarrow M$ defined by
\begin{equation*}
\phi_t(x)=\phi(t,x):=\exp_x\left(t \nabla b_v(x)\right),
\end{equation*}
for all $(t,x)\in\mathbb{R}\times S_0$.

For all $w\in T^1M$, we denote the space of orthogonal tangent vectors to $w$ by $w^{\perp}$.
Define $U=U(t):\nabla b_v(\phi_t(x))^{\perp}\rightarrow \nabla b_v(\phi_t(x))^{\perp}$ by
\begin{equation*}
U(w):=\nabla_w \nabla b_v,
\end{equation*}
for all $t\in\mathbb{R}$ and $w\in\nabla b_v(\phi_t(x))^{\perp}$.
Since $(M,g)$ is asymptotically harmonic, $\operatorname{tr}U=-\Delta b_v=h$ for some constant $h\in\mathbb{R}$.
Since $b_v$ is convex by Theorem \ref{thm_esc_from_ref}, $h\geq 0$.

\begin{lemma}\label{lem_measurez} % phi_*
For every $t\in\mathbb{R}$,
\begin{equation*}
(\phi_t)^*(d\mu_t)=e^{ht}d\mu_0.
\end{equation*}%where $\operatorname{tr}U=h$.
\end{lemma}
\begin{proof}
For all $(t,x)\in\mathbb{R}\times S_0$,
\begin{equation*}
\left.(\phi_t)_*\right|_x\left(\frac{\partial}{\partial t}\right)
=\nabla b_v(\phi_t(x)).
\end{equation*}
Then
\begin{equation}\label{eq_phit_hess}
\left.(\phi_t)_*\right|_x(w)=\left.\exp_*\right|_{t\nabla b_v(x)}\left(w,t U(w)\right).
\end{equation}
Set $x\in S_0$ and $w\in\nabla b_v(x)^{\perp}$. Define
\begin{equation*}
\Gamma(s,t):=\exp_{\sigma(s)}(t\nabla b_v(\sigma(s))),
\end{equation*}
for all $s\in (-\varepsilon,\varepsilon)$ and $t\in\mathbb{R}$,
where $\sigma=\sigma(s)$ is the curve in $S_0$ such that $\sigma(0)=x$ and $\sigma'(0)=w$ and $\varepsilon >0$.

We denote $\partial_t\Gamma=\Gamma_*\left(\frac{\partial}{\partial t}\right)$,
$\partial_s\Gamma=\Gamma_*\left(\frac{\partial}{\partial s}\right)$, and
$J(t)=\partial_s\Gamma(0,t)$, for all $t\in\mathbb{R}$.
So, $\left.(\phi_t)_*\right|_x\left(\frac{\partial}{\partial t}\right)=\nabla b_v(\phi_t(x))$ and
$\left.(\phi_t)_*\right|_x(w)=J(t)$, for all $t\in\mathbb{R}$.
Then, from equation \eqref{eq_phit_hess}, we have
\begin{equation*}
U(J(t))=\nabla_{\partial_s\Gamma(0,t)}\partial_t\Gamma=\nabla_{\partial_t\Gamma(0,t)}\partial_s\Gamma=J'(t),
\end{equation*}
for all $t\in\mathbb{R}$. Thus, we get %by choosing suitable $w$, we get
\begin{equation*}
U=A'A^{-1},
\end{equation*}
on $\nabla b_v(\phi_t(x))^{\perp}$,
where $t\in\mathbb{R}$, $R_t$ is the Jacobi operator along $t\mapsto\phi_t(x)$, and $A=A(t)$ is the $(1,1)$ tensor field on $\nabla b_v(\phi_t(x))^{\perp}$ such that $A^{\prime\prime}+R_tA=0$.
Since $(M,g)$ is asymptotically harmonic, $\operatorname{tr}U=h$ is constant. So,
\begin{equation*}
(\ln (\det A(t)))'=\operatorname{tr}U(t)=h,
\end{equation*}
and $\det A(t)=e^{ht}$. Hence,
\begin{eqnarray*}
(\phi_t)^*(d\mu_t(\phi_t(x)))&=&
d\mu_t\left(
\left.(\phi_t)_*\right|_x(w_1),\ldots,\left.(\phi_t)_*\right|_x(w_{n-1})
\right)d\mu_0(x) \\
&=& \det A(t) \: d\mu_0(x)=e^{ht}d\mu_0(x),
\end{eqnarray*}
where $w_1,\ldots,w_{n-1}$ are orthonormal tangent vectors to $S_0$ at $x$.
\end{proof}

Define $F:M\rightarrow M$ by
\begin{equation*}
F(\phi(t,x)):=\phi(\alpha(t),x),
\end{equation*}
for all $(t,x)\in\mathbb{R}\times S_0$, where $\alpha=\alpha(t)$ is a smooth function on $\mathbb{R}$ such that $\alpha(0)=t_0$.

\begin{corollary}\label{cor_F_measurez}
For every $t\in\mathbb{R}$,
\begin{equation*}
F^*(d\mu_{\alpha(t)})=e^{h\alpha(t)-ht}d\mu_t.
\end{equation*}
\end{corollary}

\begin{proof}
For all points $x\in S_t$,
\begin{equation*}
F(x)=\exp_x((\alpha(t)-t)\nabla b_v(x)),
\end{equation*}
So, similarly to Lemma \ref{lem_measurez}, Corollary \ref{cor_F_measurez} holds.
\end{proof}
Therefore, by Proposition \ref{prop_cha} and Corollary \ref{cor_F_measurez},
\begin{eqnarray*}
\int_Mf(x)d\mu(x)&=&\int_{\mathbb{R}}\int_{S_s}f(x)d\mu_s(x)ds\\
&=&\int_{\mathbb{R}}\int_{F^{-1}{(S_s)}}f(F(x))F^*(d\mu_s)(x)ds\\
&=&\int_{\mathbb{R}}\int_{S_t}f(F(x))\alpha'(t)F^*(d\mu_{\alpha(t)})(x)dt\\
&=&\int_{\mathbb{R}}\int_{S_t}f(F(x))\alpha'(t)e^{h\alpha(t)-ht}d\mu_t(x)dt\\
&=&\int_{\mathbb{R}}\int_{S_t}f(F(x))d\mu_t(x)dt\\
&=&\int_Mf(F(x))d\mu(x),
\end{eqnarray*}
if $\alpha$ is strictly monotonic and  % ``t \mapsto \alpha(t) - t"   versus  "\alpha"
\begin{equation}\label{eq_alphaz}
\alpha'(t)e^{h\alpha(t)-ht}=1
\end{equation}
for all $t\in\mathbb{R}$.
Consider two cases. First assume that $h=0$. Then, equation \eqref{eq_alphaz} holds if $\alpha(t)=t+t_0$.
Thus $F=\phi_{t_0}$.
Now assume that $h>0$. Then
equation \eqref{eq_alphaz} holds if and only if
\begin{equation*}
\alpha'(t)e^{h\alpha(t)}=e^{ht},
\end{equation*}
or, equivalently,
\begin{equation*}
\frac{e^{h\alpha(t)}-e^{ht_0}}{h}=\frac{e^{ht}-1}{h},
\end{equation*}
so
\begin{equation*}
\alpha(t)=\frac{1}{h}\ln \left(e^{ht}+e^{ht_0}-1\right).
\end{equation*}
Therefore,
\begin{equation*}
F(\phi(t,x))=\phi\left(\frac{1}{h}\ln\left(e^{ht}+e^{ht_0}-1\right),x\right).
\end{equation*}
\begin{flushright}\qedsymbol\end{flushright}
%\end{proof}

\begin{remark}
%$F$ is not an isometry since $\nabla b_{\theta}$ is an infinitesimal isometry, i.e., a Killing vector field, if and only if $(M,g)$ is flat. %cite h = 0 => flat
If $h=0$, then $F=\phi_{t_0}$ is just a translation in $\mathbb{R}^n$, so it is an isometry. On the other hand, if $h>0$, then
\begin{equation*}
F_*(\nabla b_v(\phi_t(x)))=\alpha'(t)\nabla b_v(\phi_{\alpha(t)}(x)),
\end{equation*}
so $F$ is not an isometry since
\begin{equation*}
\alpha'(t)=\frac{he^{ht}}{h(e^{ht}+e^{ht_0}-1)}=\frac{e^{ht}}{e^{ht}+e^{ht_0}-1}<1.
\end{equation*}
Also, $\alpha(t)-t$ is a strictly monotone decreasing function on $\mathbb{R}$, so $F$ is well-defined.

%$F$ is the flow of the vector field:
%\begin{equation*}
%\end{equation*}
\end{remark}

%\begin{eqnarray*}
%\int_M f(x)d\mu(x)&=&\int_{\mathbb{R}\times S_0}f(\phi(t,x))d(\phi^*\mu)(x,t)?\\ % 표기  d mu vs mu
%&=&\int_{\mathbb{R}}\int_{S_0}e^{ht}f(\phi(t,x))d\mu_0(x) dt\\
%&=&\int_{\mathbb{R}}\int_{S_0}e^{ht}f\left(F\left(\phi\left(C(t-t_0),x\right)\right)\right)d\mu_0(x)dt\\
%&=&\int_{\mathbb{R}}\int_{S_0}e^{h(C^{-1}s+t_0)}f\left(F\left(\phi\left(s,x\right)\right)\right)C^{-1}d\mu_0(x)ds\\
%&=&\int_{\mathbb{R}}\int_{S_0}e^{h(C^{-1}s+t_0)}f\left(F\left(\phi\left(s,x\right)\right)\right)C^{-1}d\mu_0(x)ds.
%\end{eqnarray*}

\section{{{Proof of Theorem \ref{thm_volz_bound}}}}%Proof of Theorem \ref{thm_volz_bound}} %Theorem \ref{thm_volz}}

%More generally, let $k\in\mathbb{Z}_+\cap [2,\ldots,\infty)$ and $v_1,v_2,\ldots,v_k\in T_p^1M$.
%Let $A=(a_{ij})\in O(k)$ such that .

For asymptotically harmonic manifolds, we can find a volume preserving mapping from two Busemann functions directly.

\begin{lemma}
Let $(M,g)$ be an asymptotically harmonic manifold, $p\in M$, and $v_1\neq v_2\in T_p^1M$.
Then the vector field $X=\nabla b_{v_1}-\nabla b_{v_2}$ is volume preserving.
\end{lemma}

\begin{proof}
Note that $X$ is non-vanishing since $v_1\neq v_2$ and
\begin{equation*}
\operatorname{div}X=-\Delta b_{v_1}+\Delta b_{v_2}.
\end{equation*}
In particular, since $(M,g)$ is asymptotically harmonic, $L_X(d\mu)=(\operatorname{div}X)d\mu=0$.
\end{proof}

Consequently, this vector field derives a flow such that it preserves $\nabla b_{v_1}$ and $\nabla b_{v_2}$, and the flow is not volume preserving in general.

\begin{lemma}\label{lem_btbt}
Let $(M,g)$ be an asymptotically harmonic manifold, $p\in M$, $v_1\neq v_2\in T_p^1M$, and
\begin{equation*}
X=\frac{1}{\|\nabla b_{v_1}-\nabla b_{v_2}\|^2}(\nabla b_{v_1}-\nabla b_{v_2}).
\end{equation*}
Then the flow $\phi_t$ of $X$ satisfies
\begin{equation*}
(\phi_t)_*(\nabla b_{v_1}(x))=\nabla b_{v_1}(\phi_t(x)), \:\:\: (\phi_t)_*(\nabla b_{v_2}(x))=\nabla b_{v_2}(\phi_t(x)),
\end{equation*}
for all $t\in\mathbb{R}$ and $x\in M$.
\end{lemma}

%\section{Proof of Theorem \ref{thm_three}}
\begin{proof}
We denote $\beta=g(\nabla b_{v_1},\nabla b_{v_2})$ and set
\begin{equation}\label{eq_Xz_def}
X:=\frac{1}{\|\nabla b_{v_1}-\nabla b_{v_2}\|^2}(\nabla b_{v_1}-\nabla b_{v_2})=\frac{1}{2-2\beta}(\nabla b_{v_1}-\nabla b_{v_2}).
\end{equation}
Then, since $g(X,\nabla b_{v_1}-\nabla b_{v_2})=1$ and $g(X,\nabla b_{v_1}+\nabla b_{v_2})=0$,
\begin{equation}\label{eq_phi_inter}
b_{v_1}(\phi_t(x))=b_{v_1}(x)+\frac{t}{2}, \:\:\: b_{v_2}(\phi_t(x))=b_{v_2}(x)-\frac{t}{2},
\end{equation}
for all $t\in\mathbb{R}$ and $x\in M$, where $\phi_t$ is the flow of $X$. In particular,
\begin{equation*}
(\phi_t)_*(\nabla b_{v_1}(x))=\nabla b_{v_1}(\phi_t(x)), \:\:\: (\phi_t)_*(\nabla b_{v_2}(x))=\nabla b_{v_2}(\phi_t(x)),
\end{equation*}
for all $t\in\mathbb{R}$ and $x\in M$.
\end{proof}
% O(N)?

%\begin{remark}
%The flow in the orthogonal direction to $\nabla b_{v_1},\ldots,\nabla b_{v_{n-1}}$'s preserve the values of $b_{v_1},\ldots,b_{v_{n-1}}$, where %$n=\dim M$.
%\end{remark}

The flow in Lemma \ref{lem_btbt} maps one intersection of two horospheres onto another intersection of two horospheres
and the intersection of two horospheres is orthogonal to each gradient of the Busemann function.
Similarly, we consider the flow of $\frac{1}{2+2\beta}(\nabla b_{v_1}+\nabla b_{v_2})$: %In particular, the orthogonal directions preserve the value of Busemann functions.

\begin{proposition}\label{prop_lower}
Let $(M,g)$ be an asymptotically harmonic, visibility manifold,
%Let $(M,g)$ be an asymptotically harmonic manifold of dimension $n\geq 3$,
$p\in M$, and $v_1\neq v_2\in T_p^1M$.
%Suppose that $(M,g)$ is a visibility manifold.
Then there exists a constant $c_0\in\mathbb{R}$ such that, for every $x\in M$,
\begin{equation*}
b_{v_1}(x)+b_{v_2}(x)\geq c_0.
\end{equation*}
In addition, if $b_{v_1}(x)+b_{v_2}(x)=c_0$, then $\nabla b_{v_1}(x)+\nabla b_{v_2}(x)=0$, for any $x\in M$.
\end{proposition}

\begin{proof}
Let $v_1\neq v_2\in T_p^1M$,
\begin{equation*}
D:=\{x\in M:\nabla b_{v_1}(x)+\nabla b_{v_2}(x)=0\},
\end{equation*}
and let $c_0$ be the value of $b_{v_1}+b_{v_2}$ on $D$. Note that $D\neq \emptyset$ since $(M,g)$ is a visibility manifold. Set
\begin{equation*}
Y:=\frac{1}{\|\nabla b_{v_1}+\nabla b_{v_2}\|^2}(\nabla b_{v_1}+\nabla b_{v_2})=\frac{1}{2+2\beta}(\nabla b_{v_1}+\nabla b_{v_2}),
\end{equation*}
on $M-D$,
where $\beta=g(\nabla b_{v_1},\nabla b_{v_2})$.
Then, since $g(Y,\nabla b_{v_1}+\nabla b_{v_2})=1$ and $g(Y,\nabla b_{v_1}-\nabla b_{v_2})=0$,
for all $s\in\mathbb{R}$ and $x\in M$ satisfying $(b_{v_1}(x)+b_{v_2}(x)-c_0)(b_{v_1}(x)+b_{v_2}(x)+s-c_0)>0$,
we have
\begin{equation}\label{eq_psi_inter}
b_{v_1}(\psi_s(x))=b_{v_1}(x)+\frac{s}{2}, \:\:\: b_{v_2}(\psi_s(x))=b_{v_2}(x)+\frac{s}{2},
\end{equation}
where $\psi_s$ is the flow of $Y$.

Now, suppose that $b_{v_1}(x)+b_{v_2}(x)=c<c_0$ for some $x\in M$. Let $U_i$ be the $(0,2)$-tensor field on $M$ defined by
\begin{equation*}
U_i(w_1,w_2):=g\left(\nabla_{w_1}\nabla b_{v_i},w_2\right)
\end{equation*}
for all $x'\in M$ and $w_1,w_2\in T_{x'}M$, where $i=1,2$. Since $b_{v_i}$ is convex, $U_i$ is positive semi-definite.
Consequently,
\begin{eqnarray}
Y[\beta]&=&\frac{(\nabla b_{v_1}+\nabla b_{v_2})\left[g\left(\nabla b_{v_1},\nabla b_{v_2}\right)\right]}{2+2\beta} \nonumber \\
&=&\frac{g\left(\nabla_{\nabla b_{v_1}}\nabla b_{v_2},\nabla b_{v_1}\right)+g\left(\nabla_{\nabla b_{v_2}}\nabla b_{v_1},\nabla b_{v_2}\right)}{2+2\beta} \nonumber \\
&=&\frac{U_2(\nabla b_{v_1},\nabla b_{v_1})+U_1(\nabla b_{v_2},\nabla b_{v_2})}{2+2\beta}\geq 0. \label{eq_Ybeta}
\end{eqnarray}
Hence, $s\mapsto \beta(\psi_s(x))$ is non-decreasing. By Lemma \ref{lem_visibility},
\begin{equation*}
(b_{v_1}+b_{v_2})^{-1}((-\infty,c_0])\cap (b_{v_1}-b_{v_2})^{-1}(t)
\end{equation*}
is compact and it contains $\psi_s(x)$ for all $s\in (0,c_0-c)$, where $t=b_{v_1}(x)-b_{v_2}(x)$.
Thus, {{there exists a sequence $s_i\in (0,c_0-c)$, $i=1,2,\ldots$, such that $\lim_{i\rightarrow\infty}s_i=c_0-c$ and the limit $x_0:=\lim_{i\rightarrow\infty}\psi_{s_i}(x)$ exists.}} Since
\begin{equation*}
(b_{v_1}+b_{v_2})(x_0)=(b_{v_1}+b_{v_2})(x)+c_0-c=c_0,
\end{equation*}
by equation \eqref{eq_intersection_in_D}, we have $x_0\in D$ so that $\beta(x_0)=-1$. Thus, since $s\mapsto \beta(\psi_s(x))$ is non-decreasing,
we have
\begin{equation*}
-1\leq\beta(x)\leq \beta(x_0)=-1,
\end{equation*}
which implies $x \in D$ and $c=c_0$. It is a contradiction to $c<c_0$.
\end{proof}

To prove Theorem \ref{thm_volz_bound}, we need the following theorem:

\begin{theorem}\label{thm_intersectionz}
Let $(M,g)$ be an asymptotically harmonic, visibility manifold.
Let $p\in M$, $v_1\neq v_2\in T_p^1M$, and $c_1,c_2\in\mathbb{R}$.
Let $S$ be the intersection $b_{v_1}^{-1}(c_1)\cap b_{v_2}^{-1}(c_2)$ of horospheres and
$\nabla b_{v_1}(x)+\nabla b_{v_2}(x)\neq 0$ for all points $x\in S$.
%Suppose that $(M,g)$ is a visibility manifold.
Then the following integrals are independent of $c_1-c_2$:
\begin{equation*}
\int_S \sqrt{\frac{1-g(\nabla b_{v_1},\nabla b_{v_2})}{1+g(\nabla b_{v_1},\nabla b_{v_2})}}d\mu', \:\:\:
\int_S \sqrt{\frac{1+g(\nabla b_{v_1},\nabla b_{v_2})}{1-g(\nabla b_{v_1},\nabla b_{v_2})}} d\mu',
\end{equation*}
where $d\mu'$ is the induced measure on the submanifold $S$ of $(M,g)$.
\end{theorem}

\begin{proof}%[Proof of Theorem \ref{thm_intersectionz}]
We adopt the notations
\begin{equation*}
\beta,X,Y,\phi_t,\psi_s,D,c_0,U_1,U_2
\end{equation*}
in the proof of Lemma \ref{lem_btbt} and Proposition \ref{prop_lower}.

Note that $D$ is closed.
If there exists a neighborhood $V$ of $x\in D$ in $M$ such that $V\subseteq D$, then, by Theorem \ref{thm_esc_from_ref},
the level set of $b_v$ containing $x$ is totally geodesic, thus
$h=-\Delta b_v=0$, for all $v\in T^1M$.
Then, since a Busemann function is convex, by the Riccati equation for horospheres, $(M,g)$ is flat.
Hence, $(M,g)$ is not a visibility manifold which contradicts the assumption of Theorem \ref{thm_intersectionz}.
%Then $\beta$ is constant and the flow in Lemma \ref{lem_btbt} is volume preserving.

Suppose that $h\neq 0$. We note that $b_{v_1}+b_{v_2}\geq c_0$ by Proposition \ref{prop_lower}.
By equations \eqref{eq_phi_inter} and \eqref{eq_psi_inter},
the flows $\phi_t,\psi_s$ map one intersection $b_{v_1}^{-1}(c_1)\cap b_{v_2}^{-1}(c_2)$ of horospheres onto another intersection, where $c_1,c_2\in\mathbb{R}$ and $c_1+c_2> c_0$.
Let $E_1,\ldots,E_n$ be a positively oriented, orthonormal frame on $M-D$, and $\theta_1,\ldots,\theta_n$ be the dual 1-forms of the frame field such that $E_1,E_2$ spans
a subbundle containing $\nabla b_{v_1}$,$\nabla b_{v_2}$, and $d b_{v_1}\wedge d b_{v_2}=\sqrt{1-\beta^2}\:\theta_1\wedge\theta_2$.
% and the orthonormal frame $E_3,\ldots,E_n$ on the intersection of horospheres is positively oriented.
%We denote the dual 1-forms of the frame field by $\theta_1,\ldots,\theta_n$.
Then $db_{v_1}\wedge db_{v_2}$ is invariant under the pullback of $\phi_t$ and $\psi_s$, so that, for all $s> 0$ and $t\in\mathbb{R}$,
\begin{eqnarray}
(\phi_t)^*(\theta_1\wedge\theta_2)&=&(\phi_t)^*\left(\frac{1}{\sqrt{1-\beta^2}}\:db_{v_1}\wedge db_{v_2}\right) \nonumber \\
&=&\frac{1}{\sqrt{1-\beta^2(\phi_t)}}\:(\phi_t)^*(db_{v_1}\wedge db_{v_2}) \nonumber \\
&=&\frac{1}{\sqrt{1-\beta^2(\phi_t)}}\:db_{v_1}\wedge db_{v_2}  \nonumber \\
&=&\sqrt{\frac{1-\beta^2}{1-\beta^2(\phi_t)}}\:\theta_1\wedge\theta_2,  \label{eq_Xz_dmu2} \\
(\psi_s)^*(\theta_1\wedge\theta_2)&=&\sqrt{\frac{1-\beta^2}{1-\beta^2(\psi_s)}}\:\theta_1\wedge\theta_2. \label{eq_Yz_dmu2}
\end{eqnarray}
on $M-D$. Since $(M,g)$ is asymptotically harmonic,
\begin{equation*}
\operatorname{div}\left(\nabla b_{v_1}-\nabla b_{v_2}\right)=0,
\:\:\: \operatorname{div}\left(\nabla b_{v_1}+\nabla b_{v_2}\right)=2h.
\end{equation*}
Thus, we obtain
\begin{eqnarray}
\operatorname{div}X&=&\operatorname{div}\left(\frac{1}{2-2\beta}\left(\nabla b_{v_1}-\nabla b_{v_2}\right)\right) \nonumber \\
&=&\left(\nabla b_{v_1}-\nabla b_{v_2}\right)\left[\frac{1}{2-2\beta}\right]+\frac{1}{2-2\beta}\operatorname{div}\left(\nabla b_{v_1}-\nabla b_{v_2}\right)\nonumber \\
&=&-\frac{\left(\nabla b_{v_1}-\nabla b_{v_2}\right)[2-2\beta]}{(2-2\beta)^2}=-\frac{X[2-2\beta]}{2-2\beta}\nonumber \\
&=&-X\left[\ln(2-2\beta)\right]\nonumber \\
&=&X\left[\ln\left(\frac{1}{1-\beta}\right)\right], \label{eq_divX}
\end{eqnarray}
and
\begin{eqnarray}
\operatorname{div}Y&=&\operatorname{div}\left(\frac{1}{2+2\beta}\left(\nabla b_{v_1}+\nabla b_{v_2}\right)\right)\nonumber \\
&=&\left(\nabla b_{v_1}+\nabla b_{v_2}\right)\left[\frac{1}{2+2\beta}\right]+\frac{1}{2+2\beta}\operatorname{div}\left(\nabla b_{v_1}+\nabla b_{v_2}\right)\nonumber \\
&=&-\frac{\left(\nabla b_{v_1}+\nabla b_{v_2}\right)[2+2\beta]}{(2+2\beta)^2}+\frac{2h}{2+2\beta}\nonumber \\
&=&-\frac{Y[2+2\beta]}{2+2\beta}+\frac{h}{1+\beta}\nonumber \\
&=&Y\left[\ln\left(\frac{1}{1+\beta}\right)\right]+\frac{h}{1+\beta}. \label{eq_divY}
\end{eqnarray}
Consequently, for the induced Riemannian measure $d\mu$ on $M-D$, the following equations hold:
\begin{eqnarray*}
\frac{\partial}{\partial t}(\phi_t)^*(d\mu)&=&(\phi_t)^*(L_X(d\mu))=(\phi_t)^*\left(\operatorname{div}Xd\mu\right)\\
&=&\left(\operatorname{div}X\circ \phi_t\right)(\phi_t)^*(d\mu),\\
\frac{\partial}{\partial s}(\psi_s)^*(d\mu)&=&\left(\operatorname{div}Y\circ\psi_s\right)(\psi_s)^*(d\mu).
\end{eqnarray*}
Set $(\phi_t)^*(d\mu)=f_t \: d\mu$ where $f_t\in C^{\infty}(M-D)$. Then
\begin{gather*}
\frac{\partial}{\partial t}f_t \: d\mu = \left(\operatorname{div}X\circ \phi_t\right) f_t\: d\mu,\\
\frac{\partial}{\partial t}\ln f_t=\operatorname{div}X\circ \phi_t.
\end{gather*}
Note that $X[f](\phi_t)=\frac{\partial}{\partial t}f(\phi_t)$ for all $f\in C^{\infty}(M-D)$. Thus, since $f_0=1$,
{{from equation \eqref{eq_divX}, we get}}
\begin{equation*}
f_t=\frac{1-\beta}{1-\beta(\phi_t)}.
\end{equation*}
Hence,
\begin{equation}\label{eq_Xz_dmu}
(\phi_t)^*(d\mu)=\frac{1-\beta}{1-\beta(\phi_t)}d\mu.
\end{equation}
Similarly, when $(\psi_s)^*(d\mu)=g_s \: d\mu$ for some $g_s\in C^{\infty}(M-D)$, {{by equation \eqref{eq_divY},}}
\begin{equation*}
\ln g_s=\ln\left(\frac{1+\beta}{1+\beta(\psi_s)}\right)+\int_0^s\frac{h}{1+\beta(\psi_k)}dk,
\end{equation*}
so
\begin{equation}\label{eq_Yz_dmu}
(\psi_s)^*(d\mu)=\exp\left(\int_0^s\frac{h}{1+\beta(\psi_k)}dk\right)\frac{1+\beta}{1+\beta(\psi_s)}d\mu.
\end{equation}
Set $d\mu':=\theta_3\wedge\cdots\wedge\theta_n$. Then, by {{equations}} \eqref{eq_Xz_dmu2}, \eqref{eq_Yz_dmu2}, \eqref{eq_Xz_dmu}, and \eqref{eq_Yz_dmu}, we have
\begin{eqnarray*}
(\phi_t)^*(d\mu')&=&\sqrt{\frac{(1-\beta)(1+\beta(\phi_t))}{(1+\beta)(1-\beta(\phi_t))}}d\mu',\\
(\psi_s)^*(d\mu')&=&\exp\left(\int_0^s\frac{h}{1+\beta(\psi_k)}dk\right)
\sqrt{\frac{(1+\beta)(1-\beta(\psi_s))}{(1-\beta)(1+\beta(\psi_s))}}d\mu',
\end{eqnarray*}
or, equivalently,
\begin{eqnarray}
\sqrt{\frac{1-\beta(\phi_t)}{1+\beta(\phi_t)}}\:(\phi_t)^*(d\mu')&=&\sqrt{\frac{1-\beta}{1+\beta}}d\mu',  \label{eq_Xz_dmun-2}\\
\sqrt{\frac{1+\beta(\psi_s)}{1-\beta(\psi_s)}}\:(\psi_s)^*(d\mu')
&=&\exp\left(\int_0^s\frac{h}{1+\beta(\psi_k)}dk\right)
\sqrt{\frac{1+\beta}{1-\beta}}d\mu'. \label{eq_Yz_dmun-2}
\end{eqnarray}

Now, let
\begin{equation*}
S(s,t):=b_{v_1}^{-1}\left(\frac{s+c_0+t}{2}\right)\cap b_{v_2}^{-1}\left(\frac{s+c_0-t}{2}\right),
\end{equation*}
for all $s\geq 0$ and $t\in\mathbb{R}$. {{Note that
\begin{equation*}
S(c_1+c_2-c_0,c_1-c_2)=b_{v_1}^{-1}(c_1)\cap b_{v_2}^{-1}(c_2),
\end{equation*}
for all $c_1,c_2\in\mathbb{R}$ such that $c_1+c_2\geq c_0$.}} Set
\begin{eqnarray*}
V(s,t)&:=&\int_{S(s,t)}\sqrt{\frac{1-\beta}{1+\beta}}d\mu',\\
W(s,t)&:=&\int_{S(s,t)}\sqrt{\frac{1+\beta}{1-\beta}}d\mu',
\end{eqnarray*}
for all $s>0$ and $t\in\mathbb{R}$.
Then $\phi_t(S(s,t_0))=S(s,t_0+t)$ and $\psi_s(S(\varepsilon,t))=S(s+\varepsilon,t)$, for all $s,\varepsilon>0$ and $t_0,t\in\mathbb{R}$.
By {{equation}} \eqref{eq_Xz_dmun-2}, {{
\begin{eqnarray}
V(s,t)&=&\int_{\phi_t(S(s,0))}\sqrt{\frac{1-\beta}{1+\beta}}d\mu' \nonumber\\
&=&\int_{S(s,0)}\sqrt{\frac{1-\beta(\phi_t)}{1+\beta(\phi_t)}}\:(\phi_t)^*(d\mu') \nonumber \\
&=&\int_{S(s,0)}\sqrt{\frac{1-\beta}{1+\beta}}d\mu' \nonumber \\
&=&V(s,0), \label{eq_Vz_t}
\end{eqnarray} }}
for all $s>0$ and $t\in\mathbb{R}$.
{{For all $s,\varepsilon >0$ and $t\in\mathbb{R}$, from equation \eqref{eq_Yz_dmun-2}, we obtain }}

\begin{eqnarray*}
W(s+\varepsilon,t)&=&\int_{\psi_s(S(\varepsilon,t))}\sqrt{\frac{1+\beta}{1-\beta}}d\mu'\\
&=&\int_{S(\varepsilon,t)}\sqrt{\frac{1+\beta(\psi_s)}{1-\beta(\psi_s)}}\:(\psi_s)^*(d\mu')\\
&=&\int_{S(\varepsilon,t)}\exp\left(\int_0^s\frac{h}{1+\beta(\psi_k)}dk\right)\sqrt{\frac{1+\beta}{1-\beta}}d\mu'.
\end{eqnarray*}
By differentiating it for $s$, {{we get %from {\red{equation}} \eqref{eq_Yz_dmun-2}, we have {\red{
\begin{eqnarray*}
\frac{\partial}{\partial s}W(s+\varepsilon,t)
&=&\int_{S(\varepsilon,t)}\frac{h}{1+\beta(\psi_s)}\exp\left(\int_0^s\frac{h}{1+\beta(\psi_k)}dk\right)\sqrt{\frac{1+\beta}{1-\beta}}d\mu'\\
&=&\int_{S(\varepsilon,t)}\frac{h}{1+\beta(\psi_s)}\sqrt{\frac{1+\beta(\psi_s)}{1-\beta(\psi_s)}}\:(\psi_s)^*(d\mu')\\
&=&\int_{S(s+\varepsilon,t)}\frac{h}{\sqrt{1-\beta^2}}d\mu'.
\end{eqnarray*} }}
By equation \eqref{eq_Vz_t}, we have
\begin{eqnarray}
\frac{\partial}{\partial t}V(s,t)&=& 0, \label{eq_Vz_diff_t}\\
\frac{\partial}{\partial s}W(s,t)&=& \int_{S(s,t)}\frac{h}{\sqrt{1-\beta^2}}d\mu' \nonumber \\
&=&\frac{h}{2}W(s,t)+\frac{h}{2}V(s,t). \label{eq_Wz_diff_s}
\end{eqnarray}
By differentiating {{equation}} \eqref{eq_Wz_diff_s} for $t$, from {{equation}} \eqref{eq_Vz_diff_t}, the following equation holds:
\begin{equation*}
\frac{\partial^2}{\partial s\partial t}W(s,t)=\frac{h}{2}\frac{\partial}{\partial t}W(s,t),
\end{equation*}
{{which implies,}} for all $s,\varepsilon >0$ and $t\in\mathbb{R}$,
\begin{equation}\label{eq_Wz_ehst}
\frac{\partial}{\partial t}W(s+\varepsilon,t)=e^{hs/2}\frac{\partial}{\partial t}W(\varepsilon,t).
\end{equation}
By {{equation}} \eqref{eq_Xz_dmun-2},
\begin{eqnarray*}
W(s,t)
&=& \int_{\phi_t(S(s,0))}\sqrt{\frac{1+\beta}{1-\beta}}d\mu'\\
&=& \int_{S(s,0)}\sqrt{\frac{1+\beta(\phi_t)}{1-\beta(\phi_t)}}\:(\phi_t)^*(d\mu')\\
&=& \int_{S(s,0)}\frac{1+\beta(\phi_t)}{1-\beta(\phi_t)}\sqrt{\frac{1-\beta(\phi_t)}{1+\beta(\phi_t)}}\:(\phi_t)^*(d\mu') \\
&=& \int_{S(s,0)}\frac{1+\beta(\phi_t)}{1-\beta(\phi_t)}\sqrt{\frac{1-\beta}{1+\beta}}d\mu'\\
&=& \int_{S(s,0)}\left(-1+\frac{2}{1-\beta(\phi_t)}\right)\sqrt{\frac{1-\beta}{1+\beta}}d\mu'.
\end{eqnarray*}
Thus we obtain
\begin{eqnarray*}
\frac{\partial}{\partial t}W(s,t)
&=& \int_{S(s,0)}\frac{2 X[\beta](\phi_t)}{(1-\beta(\phi_t))^2}\sqrt{\frac{1-\beta}{1+\beta}}d\mu'\\
&=& \int_{S(s,0)}\frac{2 X[\beta](\phi_t)}{(1-\beta(\phi_t))^2}\sqrt{\frac{1-\beta(\phi_t)}{1+\beta(\phi_t)}}\:(\phi_t)^*(d\mu')\\
&=& \int_{S(s,t)}\frac{2X[\beta]}{(1-\beta)\sqrt{1-\beta^2}}d\mu'.
\end{eqnarray*}
Now, we get, from equation \eqref{eq_Xz_def},
\begin{eqnarray}
X[\beta]&=&\frac{1}{2-2\beta}(\nabla b_{v_1}-\nabla b_{v_2})\left[g(\nabla b_{v_1},\nabla b_{v_2})\right] \nonumber \\
&=&\frac{1}{2-2\beta}\left(U_2(\nabla b_{v_1},\nabla b_{v_1})-U_1(\nabla b_{v_2},\nabla b_{v_2})\right).\label{eq_X_diff_beta}
\end{eqnarray}
Consequently, we have
\begin{eqnarray*}
\frac{\partial}{\partial t}W(s,t)
&=&\int_{S(s,t)}\frac{\sqrt{1-\beta^2}}{(1-\beta)^2}
\frac{U_2(\nabla b_{v_1},\nabla b_{v_1})-U_1(\nabla b_{v_2},\nabla b_{v_2})}
{1-g(\nabla b_{v_1},\nabla b_{v_2})^2}d\mu'\\
&=&\int_{S(s,t)}\frac{\sqrt{1-\beta^2}}{(1-\beta)^2}\frac{U_2(\nabla b_{v_1}-\beta\nabla b_{v_2},\nabla b_{v_1}-\beta\nabla b_{v_2})}{\|\nabla b_{v_1}-\beta\nabla b_{v_2}\|^2}d\mu' \\
&&-\int_{S(s,t)}\frac{\sqrt{1-\beta^2}}{(1-\beta)^2}\frac{U_1(\nabla b_{v_2}-\beta\nabla b_{v_1},\nabla b_{v_2}-\beta\nabla b_{v_1})}{\|\nabla b_{v_2}-\beta\nabla b_{v_1}\|^2}d\mu'.
\end{eqnarray*}
Now, use the following lemma:
\begin{lemma}\label{lem_bounded}
Let $(M,g)$ be an asymptotically harmonic manifold. Then, for all $v,w\in T^1M$,
\begin{equation}\label{eq_max_eig}
g(\nabla_w\nabla b_v,w)\leq h,
\end{equation}
where $h=-\Delta b_v$.
\end{lemma}

\begin{proof}
Let $U$ be the $(0,2)$-tensor field defined by $U(w_1,w_2):=g(\nabla_{w_1}\nabla b_v,w_2)$ for all $x\in M$ and $w_1,w_2\in T_xM$.
Then $U$ is symmetric, positive semi-definite, and $\operatorname{tr}U=h$.
Thus, every eigenvalue of $U$ is real, non-negative, and less than or equal to $h$. In particular, equation \eqref{eq_max_eig} holds.
\end{proof}

By Lemma \ref{lem_bounded},
\begin{equation}\label{eq_bounded_hess}
\begin{split}
&0\leq \frac{U_1(\nabla b_{v_2},\nabla b_{v_2})}{1-\beta^2}=\frac{U_1(\nabla b_{v_2}-\beta\nabla b_{v_1},\nabla b_{v_2}-\beta\nabla b_{v_1})}{\|\nabla b_{v_2}-\beta\nabla b_{v_1}\|^2}\leq h,  \\
&0\leq \frac{U_2(\nabla b_{v_1},\nabla b_{v_1})}{1-\beta^2}=\frac{U_2(\nabla b_{v_1}-\beta\nabla b_{v_2},\nabla b_{v_1}-\beta\nabla b_{v_2})}{\|\nabla b_{v_1}-\beta\nabla b_{v_2}\|^2} \leq h.
\end{split}
\end{equation}
Hence, {{by equations \eqref{eq_X_diff_beta} and \eqref{eq_bounded_hess},}} $\frac{X[\beta]}{1+\beta}$ is bounded.
By Lemma \ref{lem_visibility}, {{$\bigcup_{\varepsilon\in [0,s]}S(\varepsilon,t)$}} is compact for all $s >0$.
Fix $s>0$. Then $\beta$ has a maximum value on the compact set $\bigcup_{\varepsilon\in [0,s]}S(\varepsilon,t)$.
In particular, there exists a constant $C>0$ such that, for all $\varepsilon\in (0,s)$,
\begin{equation}\label{eq_parWz_bound}
\left|\frac{\partial}{\partial t}W(\varepsilon,t)\right|\leq C \int_{S(\varepsilon,t)}\sqrt{1+\beta}\:d\mu'.
\end{equation}
%where $\operatorname{vol}_{n-2}S(\varepsilon,t)$ is the $(n-2)$-dimensional volume of $S(\varepsilon,t)$.

%By the assumption of Theorem \ref{thm_intersectionz}, every Jacobi operator has at least two non-zero eigenvalue, so, by Theorem \ref{thm_esc_from_ref}, the $(n-2)$-dimensional volume of $S(0,t)$ equal 0.

{{Now, we use the following lemma:
\begin{lemma}\label{lem_vol_S_dec_and_betaz}
The $(n-2)$-dimensional volume $\operatorname{vol}_{n-2}S(s,t)$ of $S(s,t)$ is non-decreasing for $s>0$ where $t\in\mathbb{R}$ is fixed, and,
for all $s>0$ and $x\in S(s,t)$,
\begin{equation*}
\beta(x)\leq -2e^{-hs}+1.
\end{equation*}
\end{lemma}
}}

\begin{proof}
For all $s,\varepsilon>0$ and $t\in\mathbb{R}$,
\begin{equation*}
\begin{split}
&\operatorname{vol}_{n-2}S(s+\varepsilon,t)= \int_{\psi_s(S(\varepsilon,t))}d\mu'=\int_{S(\varepsilon,t)}(\psi_s)^*(d\mu')\\
&=\int_{S(\varepsilon,t)}\exp\left(\int_0^s\frac{h}{1+\beta(\psi_k)}dk\right)\sqrt{\frac{(1+\beta)(1-\beta(\psi_s))}{(1-\beta)(1+\beta(\psi_s))}}\:d\mu'.
\end{split}
\end{equation*}
Thus, we obtain
\begin{eqnarray*}
\frac{\partial}{\partial s}\operatorname{vol}_{n-2}S(s+\varepsilon,t)
&=& \int_{S(\varepsilon,t)}\left(\frac{h}{1+\beta(\psi_s)}-\frac{Y[\beta](\psi_s)}{1-\beta^2(\psi_s)}\right)(\psi_s)^*(d\mu')\\
&=& \int_{S(s+\varepsilon,t)}\left(\frac{h}{1+\beta}-\frac{Y[\beta]}{1-\beta^2}\right)d\mu' \\
&=& \int_{S(s+\varepsilon,t)}\frac{1}{1+\beta}\left(h-\frac{Y[\beta]}{1-\beta}\right)d\mu'
\end{eqnarray*}
Hence, we have
\begin{equation}\label{eq_S_diff_for_s}
\frac{\partial}{\partial s}\operatorname{vol}_{n-2}S(s,t)
= \int_{S(s,t)}\frac{1}{1+\beta}\left(h-\frac{Y[\beta]}{1-\beta}\right)d\mu'.
\end{equation}
By equations \eqref{eq_Ybeta} and \eqref{eq_bounded_hess},
\begin{equation}\label{eq_Ybetaz_upper}
Y\left[\ln\left(\frac{1}{1-\beta}\right)\right]=\frac{Y[\beta]}{1-\beta}\leq h,
\end{equation}
{{and we obtain }}
\begin{equation*}
\ln\left(\frac{1-\beta(x)}{1-\beta(\psi_s(x))}\right)\leq hs
\end{equation*}
and
\begin{equation*}
\beta(\psi_s(x))\leq e^{-hs}\beta(x) + 1 - e^{-hs},
\end{equation*}
for all $s>0$ and $x\in M-D$. Since, $\lim_{k\rightarrow s-}\beta(\psi_{-k}(x))=-1$ for all $x\in S(s,t)$, by considering $\beta(x)=\beta(\psi_s(\psi_{-s}(x)))$, we obtain, for all $s>0$ and $x\in S(s,t)$,
\begin{equation}\label{eq_betaz_upper}
\beta(x)\leq -2 e^{-hs}+1.
\end{equation}
Thus, by {{equations}} \eqref{eq_S_diff_for_s}, \eqref{eq_Ybetaz_upper}, and \eqref{eq_betaz_upper},
\begin{equation}%\label{eq_Sz_s_diff}
\frac{\partial}{\partial s}\operatorname{vol}_{n-2}S(s,t)
\geq\frac{1}{2(1-e^{-hs})}\left(h-h\right)\operatorname{vol}_{n-2}S(s,t)= 0.
\end{equation}
Thus, $\operatorname{vol}_{n-2}S(s,t)$ is non-decreasing for $s>0$.
\end{proof}

%By {\red{equation}} \eqref{eq_Wz_ehst} and Lemma \ref{lem_vol_S_to_0}, for all $s>0$,
%{\blue{Thus, if $\lim_{\varepsilon\rightarrow 0+}\operatorname{vol}_{n-2}S(\varepsilon,t)=\operatorname{vol}_{n-2}S(0,t)=0$, then, by the equation \eqref{eq_Wz_ehst},
By equation \eqref{eq_parWz_bound} and Lemma \ref{lem_vol_S_dec_and_betaz}, for small $\varepsilon>0$,
$\operatorname{vol}_{n-2}S(\varepsilon,t)$ is bounded, and for some constant $C>0$
\begin{equation*}
\left|\frac{\partial}{\partial t}W(\varepsilon,t)\right|\leq C\sqrt{2-2e^{-h\varepsilon}}.
\end{equation*}
Then we obtain
\begin{equation}\label{eq_Wz_ehst_limit}
\lim_{\varepsilon\rightarrow 0+}\frac{\partial}{\partial t}W(\varepsilon,t)=0,
\end{equation}
and, by equations \eqref{eq_Wz_ehst} and \eqref{eq_Wz_ehst_limit}, we have
\begin{equation}\label{eq_Wz_diff_t}
\frac{\partial}{\partial t}W(s,t)=0
\end{equation}
By equations \eqref{eq_Vz_diff_t} and \eqref{eq_Wz_diff_t}, Theorem \ref{thm_intersectionz} is proved.

%{\red{

%Since $h\neq 0$, the closure of $M-D$ is $M$. By taking a solution $X_0$ of the differential equation \eqref{eq_plz_sollll} on $M$, we can define %$F=\phi_{2c}$ on $M$ and it satisfies the equations \eqref{eq_volpreee}, \eqref{eq_twobusez_preserve}, and \eqref{eq_move_horospherez}
%by the smoothness of $\phi_{2c}$ on $M-D$.
%}}
\end{proof}

%{\red{
%\begin{proof}[Proof of Theorem \ref{thm_volz}]
%\begin{equation*}
%\frac{\partial}{\partial t}\operatorname{vol}_{n-2}S(s,t)=\int_{S(s,t)}\frac{X[\beta]}{1+\beta}d\mu'.
%\end{equation*}
%By the proof of Theorem \ref{thm_intersectionz}, there exist constants $\varepsilon >0$ and $C>0$ such that, for all $s\in (0,\varepsilon)$,
%\begin{equation*}
%\left|\frac{\partial}{\partial t}\operatorname{vol}_{n-2}S(s,t)\right|\leq C\operatorname{vol}_{n-2}S(s,t).
%\end{equation*}
%{\blue{Thus, if $\lim_{s\rightarrow 0+}\operatorname{vol}_{n-2}S(s,t)=\operatorname{vol}_{n-2}S(0,t)=0$, then
%the $(n-2)$-dimensional volume of $S(s,t)$ is independent of $t$.
%}}
%\end{proof}

%}}

%\section{Proof of Theorem \ref{thm_three}}

%\begin{lemma}
%Let $X=\sum_{i=1}^{k+2} a_i\nabla b_{v_i}$ with $a_1,\ldots,a_{k+2}\in\mathbb{R}-\{0\}$ such that $\sum_{i=1}^{k+2} a_i=0$.
%Then $X$ is volume preserving and
%\begin{equation*}
%(F_t)_*(\nabla b_{v_i}(x))=\nabla b_{v_i}(F_t(x)),
%\end{equation*}
%for all $i=1,\ldots,{k+2}$ and $x\in M$, where $F_t$ is the flow of $X$.
%\end{lemma}

%\begin{proof}
%Let $x\in M$ and $t\in\mathbb{R}$. Write
%\begin{equation*}
%c_i= b_{v_i}(F_t(x)) - b_{v_i}(x),
%\end{equation*}
%for all $i=1,\ldots,{k+2}$. For all $i=1,\ldots,k+1$,
%\begin{equation*}
%X[a_ib_{v_i}-a_{i+1}b_{v_{i+1}}]=g(X,a_i\nabla b_{v_i}-a_{i+1}\nabla b_{v_{i+1}})=a_i
%\end{equation*}

%\end{proof}

\begin{proof}[Proof of Theorem \ref{thm_volz_bound}]
Since $\beta = -1$ on $D$, the notations $X$, $\phi_t$ can be extended to $M$, and, by equation \eqref{eq_Xz_dmu}, we have
\begin{equation*}
\operatorname{vol}_{n-2}S(0,t)= \int_{S(0,0)}(\phi_t)^*(d\mu')=\int_{S(0,0)}d\mu'=\operatorname{vol}_{n-2}S(0,0).
\end{equation*}  % nabla bv1 / nabla bv2 일시theta_1  보존으로 충분.
Now, suppose that $c_1+c_2\neq c_0$.
The $(n-2)$-dimensional volume of the intersection $S=b_{v_1}^{-1}(c_1)\cap b_{v_2}^{-1}(c_2)$ satisfies
\begin{equation*}
\operatorname{vol}_{n-2}S=\int_S d\mu' \leq \int_S \frac{1}{\sqrt{1-\beta^2}}d\mu'=\frac{1}{2}(V(s,t)+W(s,t)),
\end{equation*}
where $s=c_1+c_2-c_0$, $t=c_1-c_2$, and $V(s,t)+W(s,t)$ is independent of $t$ by Theorem \ref{thm_intersectionz}.
\end{proof}

Now, we give an example which supports the main theorem.

\begin{example}
Consider the Poincar\'e upper-half plane model $\{(x,y,z)\in\mathbb{R}^3:z>0\}$ with the metric: %\cite{CFKP}:
\begin{equation*}
g=\frac{dx^2+dy^2+dz^2}{z^2}.
\end{equation*}
The distance between two points $(x_1,y_1,z_1)$, $(x_2,y_2,z_2)$ is
\begin{equation*}
d((x_1,y_1,z_1),(x_2,y_2,z_2))=2\operatorname{arcsinh}\left(\frac{\sqrt{(x_2-x_1)^2+(y_2-y_1)^2+(z_2-z_1)^2}}{2\sqrt{z_1z_2}}\right),
\end{equation*}
%where $\operatorname{arcsinh}(t)=\ln \left(t+\sqrt{t^2+1}\right)$ for all $t\geq 1$.
We consider the volume of the intersections of horospheres for the case: $v_1=\left.\frac{\partial}{\partial x}\right|_{(0,0,1)}$ and $v_2=-\left.\frac{\partial}{\partial x}\right|_{(0,0,1)}$.
The geodesic $\gamma_{v_1}$ is the unit circle centered at $(0,0,0)$ and
the equation of the geodesic sphere centered at $(\cos r,0,\sin r)$ containing $(0,0,2)$ is
\begin{equation*}
\frac{(x-\cos r)^2+y^2+(z-\sin r)^2}{4\sin r \: z}=\frac{5-4\sin r}{8\sin r},
\end{equation*}
or equivalently,
\begin{equation*}
(x-\cos r)^2+y^2+\left(z-\frac{5}{4}\right)^2=\frac{9}{16}.
\end{equation*}
Every point $(x,y,z)$ of the intersection of the geodesic spheres centered at $(\cos r,0,\sin r)$ and $(-\cos r,0,\sin r)$ containing $(0,0,2)$ satisfies
\begin{equation*}
x=0, \:\: y^2+\left(z-\frac{5}{4}\right)^2=\frac{9}{16}.
\end{equation*}
It is also the equation of the intersection of horospheres.
Write $y=\frac{3}{4}\cos t$ and $z=\frac{3}{4}\sin t+\frac{5}{4}$. Then the volume of the intersection of horospheres is
\begin{equation*}
\int_0^{2\pi} \frac{3}{3\sin t + 5}dt < 3\pi.
\end{equation*}
\end{example}

\begin{remark}
Let $(M,g)$ be an asymptotically harmonic, visibility manifold.
Let $c_1,c_2\in\mathbb{R}$ such that $c_1+c_2>c_0$ and $r>0$. The set $b_{v_1}^{-1}([c_1,c_1+r])\cap b_{v_2}^{-1}([c_2,c_2+r])$ is
 the union of countably many sets of the form
\begin{equation*}
S=\{x\in M: a_1+c_0\leq b_{v_1}+b_{v_2}\leq a_2+c_0, \:\: a_3\leq b_{v_1}-b_{v_2}\leq a_4\},
\end{equation*}
where $a_1,a_2\geq 0$ and $a_3,a_4\in\mathbb{R}$.
It can be obtained by taking the middle points of each sides of the square $[c_1,c_1+r]\times [c_2,c_2+r]$ repeatedly.
The $n$-dimensional volume of a set of the form equals
\begin{equation*}
\begin{split}
&\int_S d\mu= \int_S \theta_1\wedge\theta_2\wedge d\mu'  \\
&=\int_S db_{v_1}\wedge db_{v_2}\wedge \left(\frac{1}{\sqrt{1-\beta^2}}d\mu'\right)\\
&=\frac{1}{2}\int_{a_1}^{a_2}\int_{a_3}^{a_4}\int_{S(s,t)}\frac{1}{\sqrt{1-\beta^2}}d\mu'\:dt\:ds.
\end{split}
\end{equation*}
By Theorem \ref{thm_intersectionz},
\begin{equation*}
\int_S d\mu = \frac{a_4-a_3}{2}\int_{a_1}^{a_2}\int_{S(s,0)}\frac{1}{\sqrt{1-\beta^2}}d\mu'\:ds.
\end{equation*}
Therefore, the $n$-dimensional volume of $b_{v_1}^{-1}([c_1,c_1+r])\cap b_{v_2}^{-1}([c_2,c_2+r])$ is independent of $c_1-c_2$.
\end{remark}

%Let $\gamma$ be a unit-speed geodesic such that $\gamma(0)=p$ and $\gamma(t_0)=q$ for some $t_0>0$. Write $v_{k+1}=\gamma'(0)$ and
%$v_{k+2}=-\gamma'(0)$.

\section*{{Acknowledgements}}
%\textcolor{red}{This work was supported by the National Research Foundation of Korea(NRF) grant funded by the Korea government (MSIT) (NRF-2019R1A2C1083957).}
This work was supported by the National Research Foundation of Korea (NRF) grant funded by the Korea government (MSIT) (NRF-2019R1A2C1083957).
% supported by  .... and .... BK21


\begin{thebibliography}{KNP2}

\bibitem{BGS} W. Ballmann, M. Gromov, and V. Schroeder, Manifolds of nonpositive curvature, Progress in Mathematics, 61, Birkhäuser Boston, Inc., Boston, MA, 1985, vi+263 pp.

\bibitem{BKP} K. Biswas, G. Knieper, and N. Peyerimhoff, {\it The Fourier transform on harmonic manifolds of purely exponential volume growth}, J. Geom. Anal. {\bf 31} (2021), no. 1, 126-163.

\bibitem{BrH} M. R. Bridson and A. Haefliger, Metric spaces of non-positive curvature, Grundlehren der Mathematischen Wissenschaften, 319, Springer-Verlag, Berlin, 1999, xxii+643 pp.

%\bibitem{CFKP} J. W. Cannon, W. J. Floyd, R. Richard, and W. R. Parry, Hyperbolic geometry, Flavors of geometry, 59-115, Math. Sci. Res. Inst. Publ., 31, Cambridge Univ. Press, Cambridge, 1997.

\bibitem{CsH} B. Csik\'os and M. Horv\'ath, {\it A characterization of harmonic spaces}, J. Differential Geom. {\bf 90} (2012), no 3, 383-389.

\bibitem{CsH2} B. Csik\'os and M. Horv\'ath, {\it Harmonic manifolds and tubes}, J. Geom. Anal. {\bf 28} (2018), no. 4, 3458-3476.

\bibitem{Cha} I. Chavel, Riemannian geometry, Second edition, Cambridge Studies in Advanced Mathematics, 98, Cambridge University Press, Cambridge, 2006, xvi+471 pp.

%\bibitem{EbO} P. Eberlein and B. O'Neill, {\it Visibility manifolds}, Pacific J. Math. {\bf 46} (1973), 45-109.

\bibitem{Esc} J.-H. Eschenburg, {\it Horospheres and the stable part of the geodesic flow}, Math. Z. {\bf 153} (1977), no. 3, 237–251.

\bibitem{HaM} G. Haller and I. Mez\'c, {\it Reduction of three-dimensional, volume-preserving flows with symmetry}, Nonlinearity {\bf 11} (1998), no. 2, 319-339.

\bibitem{Heb} J. Heber, {\it On harmonic and asymptotically harmonic homogeneous spaces}, Geom. Funct. Anal. {\bf 16} (2006), no. 4, 869-890.

\bibitem{HI} E. Heintze and H.C. Im Hof, {\it Geometry of horospheres}, Jour. Diff. Geom., {\bf 12} (1977), 481-491.

\bibitem{ItS} M. Itoh and H. Satoh, {\it Information geometry of Busemann-barycenter for probability measures}, Internat. J. Math. {\bf 26} (2015), no. 6, 1541007, 26 pp. % **

\bibitem{KnP} G. Knieper and N. Peyerimhoff, {\it Harmonic functions on rank one asymptotically harmonic manifolds}, J. Geom. Anal. {\bf 26} (2016), no. 2, 750-781. % **

%\bibitem{Nak} M. Nakahara, Geometry, topology and physics, Second edition, Graduate Student Series in Physics, Institute of Physics, Bristol, 2003, xxii+573 pp.

\bibitem{Rou} F. Rouvi\'ere, {\it Radon transform on a harmonic manifold}, J. Geom. Anal. {\bf 31} (2021), no. 6, 6365-6385.

\bibitem{Sza} Z. I. Szab\`o, {\it The Lichnerowicz conjecture on harmonic manifolds}, J. Differential Geom. {\bf 31} (1990), 1-28.

\end{thebibliography}
\end{document}